\documentclass[final,1p,times]{elsarticle}
\usepackage{amsthm}
\usepackage{lmodern}
\usepackage{amssymb}
\usepackage{bm}
\usepackage{graphicx}
\usepackage{graphics}
\usepackage{caption2}
\usepackage{psfrag}
\usepackage{arydshln}
\usepackage{amsmath}
\usepackage{amscd}
\usepackage{amsfonts}
\usepackage{float}
\usepackage{latexsym}
\usepackage{color}
\usepackage{multirow}
\usepackage{lscape}
\usepackage{arydshln}
\usepackage{epstopdf}
\usepackage{lineno}

\usepackage[linesnumbered,ruled,vlined]{algorithm2e}

\newcommand{\ds}{\displaystyle}
\newcommand{\f}{\frac}

\newcommand{\om}{\Omega}
\newcommand{\p}{\partial}

\newtheorem{theorem}{Theorem}[section]

\newtheorem{lemma}[theorem]{Lemma}

\begin{document}
\begin{frontmatter}
\linenumbers

\title{Variable-order fractional wave equation: Analysis, numerical approximation, and fast algorithm}
\author[SDNU]{Jinhong Jia}
\author[SDNU]{Chuanting Jiang}

\author[WHU]{Yiqun Li\corref{mycorrespondingauthor}}
\ead{YiqunLi24@outlook.com}

\cortext[mycorrespondingauthor]{Corresponding author.}

\author[SDU]{Mengmeng Liu}

\author[SDU]{Wenlin Qiu}

\address[SDNU]{School of Mathematics and Statistics, Shandong Normal University, Shandong 250358, China}
\address[WHU]{School of Mathematics and Statistics, Wuhan University, Wuhan 430072, China}
\address[SDU]{School of Mathematics, Shandong University, Jinan 250100, China}

\begin{abstract}
We investigate a local modification of a variable-order fractional wave equation, which describes the propagation of diffusive wave in viscoelastic media
with evolving physical property.
We incorporate an equivalent formulation to prove the well-posedness of the model as well as its high order regularity estimates.
To accommodate the convolution term in the reformulated model, we  adopt the Ritz-Volterra finite element projection and  then derive the rigorous error estimate  for the fully-discretized finite element scheme. To circumvent
the high  computational cost from the temporal integral term, we exploit  the translational invariance of the discrete coefficients associated with the convolution structure and  construct a  fast divide-and-conquer algorithm which reduces the computational complexity from $O(MN^2)$ to $O(MN\log^2 N)$.
Numerical experiments are provided to verify the theoretical results and to demonstrate the accuracy and efficiency of the proposed method.
\end{abstract}

\begin{keyword}
variable-order; fractional wave equation; well-posedness; Ritz-Volterra finite element projection; fast divide-and-conquer algorithm
\end{keyword}


\pagestyle{myheadings}
\thispagestyle{plain}

\markboth{}{The variable-order fractional wave equation}

\end{frontmatter}

\section{Introduction}



In this work, we investigate a variable-order time-fractional wave equation, which describes the propagation of the diffusive wave in viscoelastic media with evolving physical property  \cite{DuSun,LiWanJia,QiuZhe,Hendy22}
\begin{equation}\label{Model0}
\begin{array}{c}
\left\{\begin{array}{l}
\ds \partial_t^2 u(\bm x,t) -K{}^R\partial_t^{\alpha(t)} \Delta u(\bm x,t) = f(\bm x,t), \quad (\bm x,t)\in \Omega\times (0,T];\\[0.05in]
\ds u(\bm x,0)=u_0(\bm x), \quad \partial_t u(\bm x,0)=\hat u_0(\bm x), \quad \bm x\in\Omega;\\[0.05in]
\ds u(\bm x,t)=0,\quad (\bm x,t)\in \partial\Omega\times [0,T],
\end{array}
\right.
\end{array}
\end{equation}
Here $\Omega\subset \mathbb R^d$ ($1\leq d \leq 3$) is a simply connected bounded domain with a piecewise smooth boundary $\partial\Omega$ and convex corners, $K$ represents the velocity of propagation in the medium, and the source term $f(\bm x,t)$, as well as the initial data $u_0(\bm x)$ and $\hat u_0(\bm x)$, are prescribed functions. The operator ${}^R\partial_t^{\alpha(t)}$ with $0< \alpha(t)<1$ denotes the variable-order Riemann-Liouville fractional derivative, which is defined as follows \cite{LorHar,Pod}
\begin{align*}
    {}^R\partial_t^{\alpha(t)} u &= \partial_t I_t^{1-\alpha(t)}u, \quad
    I_t^{\alpha(t)}u := \bigg( \frac{t^{\alpha(t)-1}}{\Gamma(\alpha(t))} \bigg)* u(\bm x,t).
\end{align*}

Due to the strong capability to model complex phenomena—such as wave propagation in viscoelastic media, acoustic attenuation in heterogeneous materials, and seismic wave behavior, extensive research has been carried out  on time-fractional wave equations, including both the constant-order \cite{HuCai,JinLaz,LiLuo,Luchko,LyuVon,Mainardi,MaiSpa,Mcl,McLean,Xu,YuaXie}  and space/time-dependent variable-order cases \cite{LiuFu,OdzMal1,OdzMal2,Hendy22,ZhaoSun,ZheCNSNS,Zheng,ZhuaLiu}. Motivated by \cite{ZheLiQiu}, the work \cite{LiWanJia}  proposes the local modification of a variable-order fractional wave equation \eqref{Model0}, which eliminates its nonphysical initial singularity and indeed is a multiscale wave model. Nevertheless, high-order regularity estimates and rigorous numerical analysis for this model remain untreated in the literature, which motivates this work.

In this work, we incorporate  an equivalent formulation of \eqref{Model0} to prove the well-posedness of the model and the high-order regularity estimates of the solutions. In terms of the numerical approximation, we adopt the Ritz-Volterra finite element projection \cite{CanLin,CheThoWah,LinThoWal,Zhang} to accommodate the convolution term in the reformulated model \eqref{Model}. In particular, we combine the estimates \eqref{k1} for the kernel function to prove the bound for the second-order temporal derivative of the error in the Ritz-Volterra projection, based on which the rigorous error estimates of the numerical scheme is derived. To circumvent
the high computational cost from the temporal integral term  \cite{FanSunWan,JiaWan}, we exploit the translational invariance of the discrete coefficients associated with the convolution structure and construct an efficient fast divide-and-conquer algorithm which reduces the computational complexity from $O(MN^2)$ to $O(MN\log^2 N)$. Numerical results are carried out to substantiate the theoretical findings and to demonstrate the efficiency of the proposed fast algorithm.
The remainder of this paper is organized as follows. In Section \ref{sec2}, we present some preliminaries and provide the theoretical analysis of the mathematical model.
Section \ref{sec3} establishes optimal error estimates for the backward Euler Ritz-Volterra finite element method and introduces the corresponding fast algorithm.
Numerical experiments that verify the correctness of the theoretical analysis and the efficiency of the proposed fast algorithm are provided in Section \ref{secnum}.

\section{Mathematical analysis}\label{sec2}
\subsection{Preliminaries}
Let $L^p(\om)$ with $1 \le p \le \infty$ be the Banach space of $p$th power Lebesgue integrable functions on $\om$. For a positive integer $m$,
let  $ W^{m, p}(\Omega)$ be the Sobolev space of $L^p$ functions with $m$th weak derivatives in $L^p(\om)$ (similarly defined with $\om$ replaced by an interval $\mathcal I$). Let  $H^m(\Omega) := W^{m,2}(\Omega)$ and $H^m_0(\Omega)$ be its subspace with the zero boundary condition up to order $m-1$. 
For a Banach space $\mathcal{X}$, let $W^{m, p}(0,T; \mathcal{X})$ be the space of functions in $W^{m, p}(0,T)$ with respect to $\|\cdot\|_{\mathcal {X}}$. All spaces are equipped with standard norms \cite{AdaFou,Eva}.

 Denote the eigenpairs of the operator $\mathcal L:=-K\Delta$ with Dirichlet boundary conditions by $\{\lambda_i^2,\phi_i\}_{i=1}^{\infty}$ in which $\{\phi_i\}_{i=1}^{\infty}$ form an orthonormal basis in $L^2(\om)$ and the eigenvalues $\{\lambda_i^2\}_{i=1}^\infty$ form a positive and  non-decreasing sequence \cite{Eva}.
  We introduce the Sobolev space $\check{H}^s(\Omega)$ for $s\geq 0$ by
$$ \check{H}^{s}(\Omega) := \big \{ q \in L^2(\Omega): \| q \|_{\check{H}^s(\Omega)}^2 : = \sum_{i=1}^{\infty} \lambda_i^{2s} (q,\phi_i)^2 < \infty \big \},$$
which is a subspace of $H^s(\Omega)$ satisfying $\check{H}^0(\Omega) = L^2(\Omega)$ and $\check{H}^2(\Omega) = H^2(\Omega) \cap H^1_0(\Omega)$ \cite{Tho}.  We use $Q$ and $Q_*$ to denote generic positive constants in which $Q$ may assume different values at different occurances.

Throughout this paper, we make the {\it Assumption} A:
$0\le \alpha(t)\le \alpha^*<1$, $|\alpha'(t)|,\,|\alpha''(t)|\le Q$ on $ [0, T]$. In addition, $\alpha(0)=\alpha'(0)=0$.

\begin{lemma}\cite{ZheLiQiu}\label{Lem:Model}
If $u$ solves \eqref{Model0}, then $u$ solves the system
\begin{equation}\label{Model}
\begin{array}{c}
 \left\{\begin{array}{l}
\p_t^2 u(\bm x,t) -K\Delta u(\bm x,t) = f(\bm x,t) + K k*\Delta u(\bm x,t), \quad (\bm x,t)\in \Omega\times (0,T];\\[0.05in]
          \ds u(\bm x,0)=u_0(\bm x),~~ \p_t u(\bm x,0)=\hat u_0(\bm x)~~\bm x\in\Omega;\\[0.05in]
          \ds u(\bm x,t)=0,\quad (\bm x,t)\in \partial\Omega\times [0,T],
     \end{array}
\right.
\end{array}
\end{equation}
 where
\begin{equation}\label{k}
\ds k(t):=\frac{d}{d t}\left(\frac{t^{-\alpha(t)}}{\Gamma(1-\alpha(t))}\right)
\end{equation}
and the following estimates hold for \eqref{k}:
\begin{equation}\label{k1}
\ds |k(t)|\le Q,\quad |k'(t)|\le Q(|\ln t|+1).
\end{equation}
\end{lemma}

Motivated by Lemma \ref{Lem:Model}, it suffices to carry out analysis for model \eqref{Model}.
For simplicity, we omit the domain $\om$ and the time interval $(0,T)$ in the Sobolev spaces and norms, e.g., we write $\| \cdot\|_{L^2(L^2)}$ instead of $\| \cdot\|_{L^2(0,T;L^2(\Omega))}$ whenever no confusion arises.

\subsection{Analysis of an auxiliary equation}
We analyze the well-posedness and regularity estimates of the solutions to an auxiliary differential equation, for some $\lambda > 0$
\begin{equation}\begin{array}{l}\label{ode}
v^{\prime \prime}  + \lambda^2 v = q - \lambda^2 (k* v),\quad t \in (0, T], \quad v(0)=v_0, \quad v^{\prime}(0)=\hat v_0.
  \end{array}
 \end{equation}

\begin{theorem}
Suppose that {\it Assumption} A holds and $q\in L^2(0,T)$, the problem (\ref{ode}) admits a unique solution $v \in H^1(0, T)$, and the following estimate holds:
\begin{equation}\label{odestab:e1}\begin{array}{l}
\ds \|v\|_{H^1(0, T)}\leq Q\big(\lambda |v_0|+ |\hat v_0| +  \|q\|_{L^2(0, T)}\big).
\end{array}
\end{equation}
In addition, if $q \in H^1(0, T)$, we have
\begin{equation}\label{odestab:e2}\begin{array}{l}
\|v\|_{H^2(0, T)}\leq Q (\lambda^2 |v_0|  + \lambda |\hat v_0| + \|q\|_{H^1(0, T)}+ \lambda  \|q\|_{L^2(0,T)}).
\end{array}
\end{equation}
Furthermore, the following estimate holds
\begin{equation}\label{odestab:e3}\begin{array}{l}
\|v\|_{H^3(0, T)}\leq Q (\lambda^3 |v_0|  + \lambda^2 |\hat v_0| + \lambda \|q\|_{H^1(0, T)}+ \lambda^2  \|q\|_{L^2(0,T)}).
\end{array}
\end{equation}
Finally, suppose $q \in H^2(0,T)$, then
\begin{equation}\label{odestab:e4}\begin{array}{l}
\|v\|_{H^4(0,T)}\le Q (\lambda^4 |v_0| + \lambda^3  |\hat v_0| + \|q\|_{H^2(0,T)} + \lambda^2 \|q\|_{H^1(0,T)} + \lambda^3 \|q\|_{L^2(0,T)}).
	\end{array}
\end{equation}
\end{theorem}
\begin{proof}
We first consider the case that $v(0) = v^{\prime}(0)=0$. Let $\mathcal X := \{z \in H^1(0,T), z(0) =z^\prime (0) = 0\}$ equipped with equivalent norm $\|z\|_{\mathcal X,\sigma} := \|e^{-\sigma t} z'\|_{L^2(0,T)}$ for some $\sigma\geq 0$. For each $v\in \mathcal X$, let $w:= \mathcal M v$ be the solution of
\begin{align}\label{ode4}
w^{\prime \prime}  + \lambda^2 w =q-\lambda^2 (k* v),\text{ for }t\in (0,T];~~w(0)=w^\prime(0)=0.
\end{align}
Then  the solution $w$ to (\ref{ode4}) could be formally expressed as \cite{Mcl}
\begin{align*}
w=\f{1}{\lambda} q * \sin(\lambda t) - \lambda (k*v)*\sin(\lambda t).
\end{align*}

To bound $\|w\|_{\mathcal X,\sigma}$, we directly differentiate the above equation to get
\begin{align}\label{w'}
w'=q * \cos(\lambda t)- \lambda (k*v^\prime)*\sin(\lambda t),
\end{align}
where the last term on the right-hand side could be reformulated as follows
\begin{equation}\label{ode:e1}\begin{array}{l}
- \lambda (k*v^\prime)*\sin(\lambda t) = -k*v^\prime+ (k*v^\prime)^{\prime}*\cos(\lambda t) \\[0.1in]
\ds \qquad   \qquad  \qquad \qquad  \,= -k*v^\prime +  k(0)v^\prime*\cos(\lambda t) + (k^\prime*v^\prime)*\cos(\lambda t).
\end{array}
\end{equation}

We incorporate \eqref{w'}-\eqref{ode:e1}, the estimate \eqref{k1}, and  apply Young's convolution inequality and the fact that
$e^{-\sigma t} [(k^\prime*v^\prime)*\cos(\lambda t)] = (e^{-\sigma t}k^\prime)*(e^{-\sigma t}v^\prime)*(e^{-\sigma t}\cos(\lambda t)) $
to bound
\begin{align}\label{ode7}
\|w\|_{\mathcal X,\sigma}\leq Q\|q\|_{L^2(0,T)}+Q\sigma^{\varepsilon-1}\|v\|_{\mathcal X,\sigma},
\end{align}
where we used the estimate
\begin{equation*}
  \ds \int_0^t e^{-\sigma y} y^{-\varepsilon} d y  =\sigma^{\varepsilon-1} \int_0^{\sigma t} e^{-z} z^{-\varepsilon} d z \leq Q\sigma^{\varepsilon-1}.
\end{equation*}
Thus, the mapping $\mathcal M:\mathcal X\rightarrow \mathcal X$ is well defined. To show its contractivity, let $w_i=\mathcal M v_i$ for $i =1$, $2$, then $e_w : = w_1- w_2$ and $e_v : = v_1- v_2$ satisfy
\begin{align}\label{ode6}
e_w^{\prime \prime}+\lambda^2 e_w   = -\lambda^2 (k*e_v).
\end{align}
Then an application of the estimate (\ref{ode7}) to (\ref{ode6}) gives
\begin{align*}
\|e_w\|_{\mathcal X,\sigma}\leq Q\sigma^{\varepsilon-1}\|e_v\|_{\mathcal X,\sigma}.
\end{align*}
Choose a sufficiently large $\sigma$ such that the mapping $\mathcal M$ is a contraction. By the Banach fixed point theorem, $\mathcal M$ has a unique fixed point $v=\mathcal Mv$ with the estimate $\|w\|_{\mathcal X,\sigma}\leq Q\|q\|_{L^2(0,T)}$.

For the problem \eqref{ode} with inhomogeneous initial conditions, we obtain
\begin{equation}\label{Sol}
 v =  v_0 \cos(\lambda t) + \f{\hat v_0}{\lambda} \sin(\lambda t) +w,
 \end{equation}
where $w$ is the fixed point of $\mathcal M$. Thus $v$ solves the problem (\ref{ode}). Differentiate the above expression to obtain
$${ v}^{\prime}=-\lambda v_0 \sin(\lambda t)  + \hat v_0 \cos(\lambda t)+w',$$
which, together with $\|w\|_{\mathcal X,\sigma}\leq Q\|q\|_{L^2(0,T)}$, leads to
\begin{equation}\label{odestab}\begin{array}{l}
\ds \|v\|_{H^1(0,T)}\leq Q\big(\lambda|v_{0}|+ |\hat v_0| + \|q\|_{L^2(0,T)}\big).
\end{array}
\end{equation}
The uniqueness of the solution to the problem (\ref{ode}) follows from that of (\ref{ode}) with $v_0 = \hat v_0 =0$. Consequently, we conclude that  (\ref{ode}) admits  a unique solution in $H^1(0,T)$ with the estimate (\ref{odestab:e1}).

To bound $v^{\prime \prime}$, we incorporate \eqref{w'}--\eqref{ode:e1} to differentiate \eqref{Sol} twice to obtain
\begin{equation}\label{vtt}\begin{array}{l}
\ds v^{\prime \prime}= -\lambda^2 v_0 \cos(\lambda t)  -\lambda \hat v_0 \sin(\lambda t) + q(0) \cos(\lambda t) + q^{\prime} * \cos(\lambda t)\\[0.1in]
\ds \qquad \qquad  - \lambda k(0) v^{\prime}*\sin(\lambda t)-\lambda (k^{\prime} * v^{\prime})*\sin(\lambda t),
\end{array}
\end{equation}
which, together with \eqref{k1}, \eqref{odestab}, the Sobolev embedding $H^1(0,T) \rightarrow L^\infty(0, T)$, and Young's inequality, gives
\begin{equation}\label{vtt:e1}\begin{array}{l}
\ds \|v^{\prime \prime}\|_{L^2} \le Q (\lambda^2 |v_0|  + \lambda |\hat v_0| + \|q\|_{H^1(0, T)} +\lambda\|v\|_{H^1(0, T)})\\[0.1in]
\ds \qquad \qquad \le Q (\lambda^2 |v_0|  + \lambda |\hat v_0| + \|q\|_{H^1(0, T)}+ \lambda  \|q\|_{L^2(0,T)}).
\end{array}
\end{equation}
We invoke this estimate with \eqref{odestab} to prove \eqref{odestab:e2}.

To bound $v^{\prime \prime \prime}$, we further differentiate \eqref{vtt} to obtain
\begin{equation}\label{vttt}\begin{array}{l}
\ds v^{\prime \prime \prime}= \lambda^3 v_0 \sin(\lambda t)  -\lambda^2 \hat v_0 \cos(\lambda t) -\lambda q(0) \sin(\lambda t) + q^{\prime}-\lambda q^{\prime} * \sin(\lambda t)\\[0.1in]
\ds \qquad \qquad  - \lambda^2 k(0) v^{\prime}*\cos(\lambda t)  -\lambda^2 (k^{\prime} * v^{\prime})*\cos(\lambda t).
\end{array}
\end{equation}
We combine \eqref{k1}, \eqref{odestab}, and \eqref{vtt:e1} to obtain
\begin{equation*}
\begin{array}{l}
\ds \| v^{\prime \prime \prime}\|_{L^2} \le Q (\lambda^3 |v_0|  + \lambda^2 |\hat v_0| + \lambda \|q\|_{H^1(0, T)} + \lambda^2\|v\|_{H^1(0, T)})\\[0.1in]
\ds \qquad \qquad \le Q (\lambda^3 |v_0|  + \lambda^2 |\hat v_0| + \lambda \|q\|_{H^1(0, T)}+ \lambda^2  \|q\|_{L^2(0,T)}).
\end{array}
\end{equation*}
To bound $v^{\prime \prime \prime \prime}$, we further differentiate \eqref{vttt} to obtain
\begin{equation*}
\begin{array}{l}
\ds v^{\prime \prime \prime \prime}= \lambda^4 v_0 \cos(\lambda t) + \lambda^3 \hat v_0 \sin(\lambda t) - \lambda^2 q(0) \cos(\lambda t) + q^{\prime \prime} - \lambda^2 q^{\prime} * \cos(\lambda t) \\[0.1in]
\ds \quad \quad \quad- \lambda^2  k(0) v^{\prime} + \lambda^3 k(0) v^{\prime} * \sin(\lambda t) -\lambda^2 (k^{\prime} * v^{\prime}) +\lambda^3 (k^{\prime} * v^{\prime}) * \sin(\lambda t).
	\end{array}
\end{equation*}
We combine \eqref{k1}, \eqref{odestab}, and \eqref{odestab:e2} to obtain
\begin{equation*}
\begin{array}{l}
\ds \|v^{\prime \prime \prime \prime}\|_{L^2} \le Q (\lambda^4 |v_0| + \lambda^3 |\hat v_0| + \|q\|_{H^2(0,T)} + \lambda^2 \|q\|_{H^1(0,T)} + \lambda^3 \|v\|_{H^1}) \\[0.1in]
\ds \qquad \quad \le Q(\lambda^4 |v_0| + \lambda^3 |\hat v_0| + \|q\|_{H^2(0,T)} + \lambda^2 \|q\|_{H^1(0,T)} + \lambda^3 \|q\|_{L^2}(0,T)).
	\end{array}
\end{equation*}
This finishes the proof.
\end{proof}

\subsection{Analysis of model (\ref{Model})}

We now prove the well-posedness and solution regularity of the problem \eqref{Model} to support the error estimate of its numerical scheme.
\begin{theorem}\label{thm1:pde}
Suppose  {\it Assumption} A holds and that $f \in L^2(L^2)$, $u_{0} \in \check H^1$, and $\hat u_0 \in L^2$,  then the problem (\ref{Model})  has a unique solution $u \in H^1(L^2)$ with
\begin{equation}\label{pdestab1}\begin{array}{l}
\ds \|u\|_{H^1(L^2)}\leq Q\big(\|u_{0}\|_{\check H^1}+ \|\hat u_{0}\|_{L^2}+ \|f\|_{L^2(L^2)}\big).
\end{array}
\end{equation}
Suppose that $f \in H^1(L^2) \cap L^2(\check H^1)$, $u_{0} \in \check H^2$, and $\hat u_0 \in \check H^1$,  then the following estimate holds
\begin{equation}\label{pdestab2}\begin{array}{l}
\ds \|u\|_{H^2(L^2)} + \|u\|_{L^2(\check H^2)}\leq Q\big(\|u_{0}\|_{\check H^2}+ \|\hat u_{0}\|_{\check H^1}+ \|f\|_{H^1(L^2)} + \|f\|_{L^2(\check H ^1)}\big).
\end{array}
\end{equation}
In addition, suppose that $f \in H^1(\check H^1) \cap L^2(\check H^2)$, $u_{0} \in \check H^3$, and $\hat u_0 \in \check H^2$,  we have
\begin{equation}\label{pdestab3}\begin{array}{l}
\|u\|_{H^3(L^2)}  +\| u\|_{H^1(\check H^2)}  \leq Q\big(\|u_{0}\|_{\check H^3}+ \|\hat u_{0}\|_{\check H^2}+ \|f\|_{H^1(\check H^1)} + \|f\|_{L^2(\check H ^2)}\big).
\end{array}
\end{equation}
Furthermore, suppose that $f \in H^2(\check{H}^2) \cap L^2(\check{H}^3)$, $u_0 \in \check{H}^4$, and $\hat u_0 \in \check{H}^3$, we have
\begin{equation}\label{pdestab4}\begin{array}{l}
\|u\|_{H^4(L^2)}  +\| u\|_{H^2(\check H^2)}  \leq Q\big(\|u_{0}\|_{\check H^4}+ \|\hat u_{0}\|_{\check H^3}+ \|f\|_{H^2(\check H^2)} + \|f\|_{L^2(\check H ^3)}\big).
	\end{array}
\end{equation}
\end{theorem}
\begin{proof}
For $t \in [0, T]$, we expand $u$ and $f$ in \eqref{Model} with respect to  $\{\phi_i\}_{i=1}^{\infty}$ as follows \cite{SakYam,ZheWanSICON}
$$u=\sum_{i=1}^\infty u_i(t)\phi_i(\bm x),~u_i(t) := \big (u(\cdot,t),\phi_i \big),~~f=\sum_{i=1}^\infty f_i(t)\phi_i(\bm x),~f_i(t) := \big (f(\cdot,t),\phi_i \big).$$
We invoke these expressions into (\ref{Model}) to find that u solves \eqref{Model} if and only if $\{u_i\}_{i=1}^{\infty}$ satisfy the following integro-differential equations for $i=1,2, \cdots$
\begin{align}\label{ode1}
u_i^{\prime \prime}+\lambda_i^2 u_i = f_i -\lambda_i^2 (k* u_i),\quad t\in (0,T]
\end{align}
with $u_i(0)=u_{0,i}:=(u_0,\phi_i)$, and  $u_i^{\prime}(0)=\hat u_{0,i}:=(\hat u_0,\phi_i)$.

We first invoke  \eqref{odestab:e1} to bound $\bar u : = \int_0^t \sum_{i=1}^\infty u_i^{\prime}(s) \phi_i(\bm x) ds + u_0$ by
\begin{equation}\label{PDE:e1}\begin{split}
\ds \|\bar u\|_{H^1(L^2)}^2& \leq Q \|\p_t \bar u\|^2_{L^2(L^2)} = \sum_{i=1}^\infty \|u_i^{ \prime}\|_{L^2(0,T)}^2\\
&\hspace{-0.1in}\ds \leq  Q\sum_{i=1}^\infty \big(\lambda_i^2|u_{0,i}|^2 + |\hat u_{0,i}|^2 +\|f_i\|_{L^2(0,T)}^2\big)\\
&\hspace{-0.1in}\ds  =Q\big(\|u_0\|_{\check H^1}^2 +\|\hat u_0\|_{L^2}^2  +\|f\|^2_{L^2(L^2)}\big).
\end{split}
\end{equation}
As $u_i(t) = \int_0^t u_i'(s)ds + u_{0,i}$ satisfies the differential equation in (\ref{ode1}) for $i \ge 1$, we conclude that $\bar u \in H^1(L^2)$ is a solution to problem (\ref{Model}).
The uniqueness of the solution to the problem (\ref{Model})   follows from that of the differential equation in (\ref{ode1}). We thus prove the first statement of the theorem.

We further combine   (\ref{odestab:e2})  and follow the procedures in (\ref{PDE:e1})  to obtain
\begin{equation*}\begin{split}
\ds  \|\p_t^2u  \|^2_{L^2(L^2)} &  = \sum_{i=1}^\infty  \| u_i^{\prime \prime}\|_{L^2(0,T)}^2\\
&\leq Q\sum_{i=1}^\infty \big( \|f_i\|^2_{H^1(0,T)} + \lambda_i^2 \|f_i\|^2_{L^2(0,T)} + \lambda_i^4 |u_{0,i}|^2+\lambda_i^2 |\hat u_{0,i}|^2 \big)\\
&\ds \le Q \big ( \|f\|^2_{H^1(L^2)} +\|f\|^2_{L^2(\check H^1)}+ \|u_0\|^2_{\check H^2}+\|\hat u_0\|^2_{\check H^1} \big),
\end{split}\end{equation*}
which, together with \eqref{pdestab1}, gives the estimate for $\|u\|_{H^2(L^2)}$ in \eqref{pdestab2}.

To bound $\mathcal{L} u$, we incorporate the equation \eqref{Model}, \eqref{k1} and Young's inequality to obtain
\begin{equation*}\begin{array}{l}
\|e^{-\sigma t} \mathcal{L} u\|_{L^2(L^2)} = \|e^{-\sigma t}(f - \p_t^2u -  k*\mathcal{L} u)\|_{L^2(L^2)}\\[0.1in]
\qquad  \le Q(\|u\|_{H^2(L^2)} + \|f\|_{L^2(L^2)} + \sigma^{\varepsilon-1} \|e^{-\sigma t}\mathcal{L} u\|_{L^2(L^2)}).
 \end{array}\end{equation*}
 Choose a sufficient large $\sigma$ to cancel the last term on the right-hand side of the above inequality to obtain
 \begin{equation*}\begin{array}{l}
\|u\|_{L^2{(\check H^2)}} \le Q \|\mathcal{L} u\|_{L^2(L^2)} \le Q(\|u\|_{H^2(L^2)} + \|f\|_{L^2(L^2)})\\[0.1in]
\ds \qquad \qquad \qquad \le Q\big(\|u_{0}\|_{\check H^2}+ \|\hat u_{0}\|_{\check H^1}+ \|f\|_{H^1(L^2)} + \|f\|_{L^2(\check H ^1)}\big).
 \end{array}\end{equation*}

 We combine \eqref{odestab:e3}, \eqref{pdestab2} and follow the steps in \eqref{PDE:e1} to obtain
\begin{equation}\label{PDE:e2}\begin{array}{l}
\hspace{-0.15in}\ds  \|u  \|_{H^3(L^2)} \le Q \big(\|u_{0}\|_{\check H^3}+ \|\hat u_{0}\|_{\check H^2}+ \|f\|_{H^1(\check H^1)} + \|f\|_{L^2(\check H ^2)}\big).
\end{array}\end{equation}
We then invoke \eqref{pdestab2}, \eqref{PDE:e2}, the Young's inequality and the relation $\p_t \mathcal{L} u =  \p_t f - \p_t^3 u - k(0)\mathcal{L} u - k^{\prime}*\mathcal{L} u$ from \eqref{Model} to obtain
 \begin{equation}\label{PDE:e3}\begin{array}{l}
\hspace{-0.15in}\ds  \|\p_t \mathcal{L} u   \|_{L^2(L^2)} \le \| \p_t f - \p_t^3 u - k(0)\mathcal{L} u - k^{\prime}*\mathcal{L} u\|_{L^2(L^2)}\\[0.1in]
\ds \qquad \qquad \qquad \le Q \big(\|u_{0}\|_{\check H^3}+ \|\hat u_{0}\|_{\check H^2}+ \|f\|_{H^1(\check H^1)} + \|f\|_{L^2(\check H ^2)}\big).
\end{array}\end{equation}
We  combine \eqref{pdestab2} and  \eqref{PDE:e2}--\eqref{PDE:e3} to prove \eqref{pdestab3}.
We combine \eqref{odestab:e4}, \eqref{pdestab3} and follow the steps in \eqref{PDE:e1} to obtain
\begin{equation}\label{PDE:e4}\begin{array}{l}
\hspace{-0.15in}\ds  \|u\|_{H^4(L^2)} \le Q (\|u_0\|_{\check H^4} + \|\hat u_0\|_{\check H^3} + \|f\|_{H^2(L2)} + \|f\|_{H^1(\check H^2)} + \|f\|_{L^2(\check H^3)})\\[0.1in]
\ds \qquad \qquad \le Q(\|u_0\|_{\check H^4} + \|\hat u_0 \|_{\check H^3} +\|f\|_{H^2(\check H^2)} + \|f\|_{L^2(\check H^3)}).
\end{array}\end{equation}
We then invoke  \eqref{k1}, \eqref{pdestab3}, \eqref{PDE:e4}, Young's inequality and the relation $\p_t^2 \mathcal{L} u =\p_t^2 f - \p_t^4 u - k(0) \p_t \mathcal{L} u - k^{\prime} \mathcal{L}u_0 - k^{\prime}*\p_t\mathcal{L} u$ from \eqref{Model} to obtain
\begin{equation}\label{PDE:e5}\begin{array}{l}
\hspace{-0.15in}\ds  \|\p_t^2 \mathcal{L} u   \|_{L^2(L^2)} \le \|\p_t^2 f - \p_t^4 u - k(0) \p_t \mathcal{L} u - k^{\prime} \mathcal{L}u_0 - k^{\prime}*\p_t\mathcal{L} u\|_{L^2(L^2)}\\[0.1in]
\ds \qquad \qquad \quad \le Q \big(\|u_{0}\|_{\check H^4}+ \|\hat u_{0}\|_{\check H^3}+ \|f\|_{H^2(\check H^2)} + \|f\|_{L^2(\check H ^3)}\big).
\end{array}\end{equation}
We  combine \eqref{pdestab3} and  \eqref{PDE:e4}--\eqref{PDE:e5} to prove \eqref{pdestab4}.
We thus complete the proof.
\end{proof}

\section{A finite element approximation and its fast algorithm}\label{sec3}

\subsection{A Ritz-Volterra projection}

We consider a quasi-uniform partition of $\Omega$ with mesh parameter $h$, and let $S_h$ be the space of continuous,
piecewise linear functions on $\Omega$ with respect to this partition. We define a bilinear
form $a(\cdot,\cdot)$ as
\begin{equation*}
a(u,\chi):=({K}\nabla u,\nabla \chi),\quad\forall\chi \in H^1_0(\Omega).
\end{equation*}
Let $R_h:H^1_0(\Omega)\to S_h(\Omega)$ be the Ritz projection defined by
\begin{equation*}
a(u-R_hu,\chi_h)=0,\quad \forall \chi_h \in S_h.
\end{equation*}
Then $R_hu$ satisfies the approximation property \cite{Tho}
\begin{equation*}
\begin{array}{ll}
\ds \| u-R_hu \|+h\| u-R_hu\|_1\le Qh^2\| u\|_2, \quad \forall u\in H^2(\Omega)\cap H_0^1(\Omega).
\end{array}
\end{equation*}
By \eqref{k1}, we follow \cite{CanLin,CheThoWah,LinThoWal} to define the Ritz-Volterra projection $V_h:L^1(H^1_0)\to L^1(S_h)$ as follows
\begin{equation}\label{Def:RV}
a((V_hu-u)(t),\chi_h)+k(t)*a((V_hu-u)(t),\chi_h)=0,\quad \forall \chi_h \in S_h.
\end{equation}
Define $\rho:=V_hu-u$. We refer to some lemmas from  \cite{CanLin,CheThoWah,LinThoWal} for future use.
\begin{lemma}\label{thm:Vh}
For $u\in L^1(H_0^1)$, there exists the Ritz-Volterra projection $V_hu\in L^1(S_h)$ such that
\begin{equation*}
\begin{split}
\|V_hu\|_{L^1(H^1)}\le Q\|u\|_{L^1(H^1)}.
\end{split}
\end{equation*}
\end{lemma}

\begin{lemma}\label{thm:RV1}
Suppose $u \in W^{1,1}(\check H^2)$ and $u_0 \in \check H^2$, then the following estimate holds for  $ 0\le t \le T$
\begin{equation*}
\|\rho(\cdot,t)\|+h\|\rho(\cdot,t)\|_{\check H^1}\le Qh^2\big(  \|u_0\|_{\check H^2}
+ \|u\|_{W^{1,1}(0,t;\check H^2)} \big).
\end{equation*}
In addition, suppose $\hat u_0 \in \check H^2$ and $u \in W^{2,1}(\check H^2)$, then we have
\begin{equation*}
\begin{array}{c}
\ds
\|\p_t\rho(\cdot,t)\|+h\|\p_t\rho(\cdot,t)\|_{\check H^1}\le Qh^2\left(\|u_0\|_{\check H^2}+\|\hat u_0\|_{\check H^2}+\| u\|_{W^{2,1}(0,t;\check H^2)}\right).
\end{array}
\end{equation*}
\end{lemma}

We subsequently prove the following bound for the second-order temporal derivative of the error in the Ritz–Volterra projection.

{\begin{theorem}\label{thm4}
Suppose that $u_0$, $\hat u_0 \in \check H^2$ and $u \in W^{2,1}(\check H^2)$, we have the following estimate
\begin{equation*}
\begin{split}
\|\p_t^2\rho \|_{L^1(0,t;L^2)} + h\|\p_t^2\rho\|_{L^1(0,t;\check H^1)}
&\ds \le Qh^2\left(\|u_0\|_{\check H^2}+\|\hat u_0\|_{\check H^2}+\| u\|_{W^{2,1}(0,t;\check H^2)}\right).
\end{split}
\end{equation*}
\end{theorem}
}
\begin{proof}
By differentiating (\ref{Def:RV}) twice, we obtain
\begin{equation}\label{thm4:e2}
a(\partial_t^2\rho,\chi_h)=-k'(t)a(\rho(\cdot,0),\chi_h)-k(t)a(\p_t\rho(\cdot,0),\chi_h)-k(t)*a(\p_t^2\rho,\chi_h).
\end{equation}
The $H^1$-norm can be estimated from
\begin{equation*}
\begin{split}
&c_0 \|\p_t^2(V_hu-R_hu)\|^2_{\check H^1}
\le a\left(\p_t^2(V_hu-R_hu),\p_t^2(V_hu-R_hu)\right)\\
&= a(\p_t^2\rho, \p_t^2(V_hu-R_hu))+ a(\p_t^2(u-R_hu),\p_t^2(V_hu-R_hu))\\
 &\le Q\left(\vert k'(t)\vert \|\rho_0\|_{\check H^1}+\|\p_t\rho(\cdot,0)\|_{\check H^1}
+\int_0^t \vert k(t-s)\vert \|\p_s^2\rho\|_{\check H^1}ds \right)\|\p_t^2(V_hu-R_hu)\|_{\check H^1}.
\end{split}
\end{equation*}
By eliminating $\|\p_t^2(V_hu-R_hu)\|_{\check H^1}$ on both sides and using the Cauchy inequality, we obtain
\begin{equation*}\begin{split}
&\|\partial_t^2\rho\|_{\check H^1}\le \|\partial_t^2(R_hu-u)\|_{\check H^1}+\|\p_t^2(V_hu-R_hu)\|_{\check H^1}\\
&\le Q\left(h\|\partial_t^2u\|_{\check H^2}+\vert k'(t)\vert \|\rho_0\|_{\check H^1}+\|\p_t\rho(\cdot,0)\|_{\check H^1}\right)
+\int_0^t \vert k(t-s)\vert \|\p_s^2\rho\|_{\check H^1}ds.
\end{split}
\end{equation*}
Note that $\rho(\cdot,0)=R_hu-u_0$ and $\partial_t\rho(\cdot,0)=R_h\hat u_0-\hat u_0$. By applying the Gronwall inequality and Lemma \ref{thm:RV1}, we have
\begin{equation}\label{thm4:e3}
\begin{split}
\|\partial_t^2\rho(\cdot, t)\|_{\check H^1}\le Qh\left(t^{-\epsilon}\|u_0\|_{\check H^2}+\|\hat u_0\|_{\check H^2}+\|\p_t^2u(\cdot,t)\|_{\check H^2}\right).
\end{split}
\end{equation}

We estimate the $L^2$-norm of $\p_t^2\rho$ by duality argument.  Let $\psi$ be the solution of $\ds -K \Delta \psi(\bm{x})=\phi(\bm{x}), \bm{x}\in \Omega$ and $\psi|_{\partial\Omega}=0$ with $\|\phi\|=1$, we have
\begin{equation}\label{thm4:e4}
\|\p_t^2 \rho\|=\sup_{\|\phi\|=1}(\p_t^2 \rho,\phi)=a(\p_t^2 \rho,\psi-R_h\psi)+a(\p_t^2 \rho,R_h\psi).
\end{equation}
We use (\ref{thm4:e2}) to bound the second term on the right-hand side of (\ref{thm4:e4}) by
\begin{equation}\label{thm4:e5}
\begin{split}
\vert a(\p_t^2 \rho,R_h\psi)\vert
&=\bigg|-k'(t)a(\rho(\cdot,0),R_h\psi)-k(t)a(\p_t\rho(\cdot,0),R_h\psi)\\
&\quad-\int_{0}^{t}k(t-s)a(\p_s^2\rho(\cdot,s),R_h\psi)ds\bigg|\\
&\displaystyle=\bigg|-k'(t)a(\rho(\cdot,0),R_h\psi-\psi)-k'(t)a(\rho(\cdot,0),\psi)\\
&\quad-k(t)a(\p_t \rho(\cdot,0),R_h\psi-\psi)-k(t)a(\p_t \rho(\cdot,0),\psi)\\
&\quad-\int_{0}^{t}k(t-s)a(\p_s^2\rho,R_h\psi-\psi)ds-\int_{0}^{t}k(t-s)a(\p_s^2\rho,\psi)ds\bigg|\\
&\displaystyle=\bigg|-k'(t)a(\rho(\cdot,0),R_h\psi-\psi)
+  K k'(t)(\rho(\cdot,0),\Delta\psi)\\
&\quad-k(t)a(\p_t\rho(\cdot,0),R_h\psi-\psi)+ K k(t)(\p_t\rho(\cdot,0),\Delta\psi)\\
&\quad-\int_{0}^{t}k(t-s)a(\p_s^2\rho,R_h\psi-\psi)ds+K\int_{0}^{t}k(t-s)(\p_s^2\rho,\Delta\psi)ds\bigg|\\
&\ds\le Q\left(t^{-\epsilon}(\|\rho(\cdot,0)\|+h\|\rho(\cdot,0)\|_{\check H^1})
+(\|\rho_t(\cdot,0)\|+h\|\rho_t(\cdot,0)\|_{\check H^1})\right)\\
&\quad \ds +Qh\int_0^t\|\p_s^2\rho\|_{\check H^1}ds+Q\int_0^t\|\p_s^2\rho\|ds.
\end{split}
\end{equation}
Substituting (\ref{thm4:e3}) and (\ref{thm4:e5}) into (\ref{thm4:e4}), and applying the Gronwall inequlity, we obtain
\begin{equation}\label{thm4:e6}
\begin{split}
\|\p_t^2\rho(\cdot, t)\|
&\ds\le Ch^2\bigg(t^{-\epsilon}\|u_0\|_{\check H^2}
+\|\hat u_0\|_{\check H^2} + \int_0^t\|\partial_s^2 u\|_{\check H^2}\,ds\bigg).
\end{split}
\end{equation}
We combine (\ref{thm4:e3}) and (\ref{thm4:e6}) to finish the proof.
\end{proof}

\subsection{Analysis of semi-discrete scheme}

The weak formulation of (\ref{Model}) can be written as
\begin{equation}\label{Vari0}
\begin{gathered}
(\partial_{t}^2u,\chi)+a(u,\chi)+\int_{0}^{t}k(t-s)a(u(\cdot,s),\chi)ds=(f,\chi), \\
u(\cdot,0)=u_0,\quad \partial_tu(\cdot,0)=\hat{u}_0,\quad \forall \chi \in H_0^1.
\end{gathered}
\end{equation}
Accordingly, the semi-discrete finite element approximation to (\ref{Model}) seeks $u_h(t):[0,T]\rightarrow S_h$ such that
\begin{equation}\label{Vari}
\begin{gathered}
(\partial_t^2u_h,\chi_h)+a(u_h,\chi_h)+\int_{0}^{t}k(t-s)a(u_h(s),\chi_h)ds=(f,\chi_h), \\
u_h(\cdot,0)=R_hu_0,\quad \partial_tu_{h}(\cdot,0)=R_h\hat{u}_{0},\quad \forall \chi_h \in S_h.
\end{gathered}
\end{equation}

\begin{theorem}
Suppose that {\it Assumption} A holds and that $f \in H^2(\check{H}^2) \cap L^2(\check{H}^3)$, $u_0 \in \check{H}^4$, and $\hat u_0 \in \check{H}^3$, then the semi-discrete solution $u_h$ satisfies the following error estimate
\begin{equation*}
\begin{split}
\| u_h(t)-u(t)\|&\ds
\le Qh^2\left(\|u_{0}\|_{\check H^4}+ \|\hat u_{0}\|_{\check H^3}+ \|f\|_{H^2(\check H^2)} + \|f\|_{L^2(\check H ^3)}\right).
\end{split}
\end{equation*}
\end{theorem}
\begin{proof}
Let $u_h - u = (u_h - V_hu) + (V_hu - u) = \theta + \rho.$
By taking $\chi = \chi_h \in S_h$ in (\ref{Vari0}) and subtracting it from (\ref{Vari}), we obtain
\begin{equation}\label{thm:sem2}
(\partial_t^2\theta,\chi_h) + a(\theta,\chi_h)
+ \int_{0}^{t} k(t-s)a(\theta(s),\chi_h)\,ds
= -(\partial_t^2\rho,\chi_h).
\end{equation}
Setting $\chi_h = \partial_t\theta$ in (\ref{thm:sem2}) yields
\begin{equation}\label{thm:sem3}
\begin{split}
\frac{1}{2}\frac{d}{dt}\!\left(\|\partial_t\theta\|^2 + a(\theta,\theta)\right)
&= -\frac{d}{dt}\int_{0}^{t} k(t-s)a(\theta(\cdot,s),\theta(\cdot,t))\,ds
+ k(0)a(\theta,\theta) \\
&\quad + \int_{0}^{t} k'(t-s)a(\theta(\cdot,s),\theta(\cdot,t))\,ds
- (\partial_t^2\rho,\partial_t\theta).
\end{split}
\end{equation}
Integrating (\ref{thm:sem3}) from $0$ to $t$, and using Cauchy’s inequality together with the fact that $\theta(0)=0$, we obtain
\begin{equation*}
\begin{split}
\|\partial_t\theta\|^2 + c_0\|\theta\|_{\check H^1}^2
&\le \|\partial_t\theta(\cdot,0)\|^2
+ 2\int_{0}^{t}\|\partial_s^2\rho\|\,\|\partial_s\theta\|\,ds
+ Q\int_{0}^{t}\|\theta(\cdot, s)\|_{\check H^1}^2\,ds.
\end{split}
\end{equation*}
Applying the generalized Gronwall inequality, we have
\begin{equation*}
\begin{split}
\|\partial_t\theta\|^2 + \|\theta\|_{\check H^1}^2
&\le Q\!\left(
\|\partial_t\theta(\cdot,0)\|^2
+ \int_{0}^{t}\|\partial_s^2\rho\|\,\|\partial_s\theta\|\,ds
\right) \\
&\le Q\!\left(
\|\partial_t\theta(\cdot,0)\|
+ \int_{0}^{t}\|\p_s^2\rho\|\,ds
\right)
\sup_{0\le s\le t}\|\partial_s\theta(s)\|.
\end{split}
\end{equation*}

Since $\partial_tu_{h,0}=R_h\hat u_0$, we have $\|\p_t\theta(\cdot,0)\|\le Ch^2\|\hat u_0\|_{\check H^2}.$
Using (\ref{thm4:e6}), we finally obtain
\begin{equation*}
\|\partial_t\theta\|
\le Qh^2\!\left(
\|u_0\|_2 + \|\hat{u}_0\|_2 + \int_0^t\|\p_s^2u\|_{\check H^2} \, ds
\right).
\end{equation*}
Note that
\begin{equation*}\begin{split}
\|\theta(\cdot,t)\|\le \|\theta_0\|_{L^2}+\int_0^t\|\partial_s\theta(\cdot,s)\|\,ds
\le Qh^2\left(\|u_0\|_{\check H^2}+\|\hat u_0\|_{\check H^2}+\|\partial_s^2u\|_{L^1(0,t;\check H^2)}\right).
\end{split}
\end{equation*}
Together with Lemma \ref{thm:RV1} and Theorem \ref{thm1:pde}, the desired result follows.
\end{proof}

\subsection{The backward Euler scheme}
For a positive integer $N$, we partition the interval $[0,T]$ by defining $t_n:=n\tau$ for $0\le n \le N$, where $\tau:=T/N$. For simplicity, denote $u_n:=u(\bm{x},t_n)$ , $f_n:=f(\bm{x},t_n)$, and define the backward difference operator $\displaystyle\delta_\tau u_n=\frac{u_{n}-u_{n-1}}{\tau}$. We approximate $\partial_t^2u$ at
$t=t_n$ using the backward Euler scheme as follows:
\begin{equation}\label{Euler:e1}
\begin{split}
\partial_t^2u_n&=\frac{u_n-2u_{n-1}+u_{n-2}}{\tau^2}+E_n^1:=\delta_\tau^2u_n+E_n^1,
\end{split}
\end{equation}
where
\begin{equation}\label{Euler:loc1}
\begin{split}
E_n^1:
&\ds=\frac{1}{\tau}\int_{t_{n-1}}^{t_n}\left(\partial_s^3u(s)
+\frac{1}{\tau}\int_{s-\tau}^s\partial_y^3u(y)dy\right)(s-t_{n-1})ds.
\end{split}
\end{equation}

Substituting (\ref{Euler:e1}) into (\ref{Vari}), we rewrite it as
\begin{equation}\label{Euler:e2}
\ds (\delta_\tau^2u_n,\chi_h)+a(u_n,\chi_h)=-k(t)*a(u,\chi_h)|_{t=t_n}+(f_n,\chi_h)-(E_n^1,\chi_h).
\end{equation}
{\begin{lemma}\label{lemma:loc1}
Assume that Assumption A holds and that $f \in H^2(\check{H}^2) \cap L^2(\check{H}^3)$, $u_0 \in \check{H}^4$, and $\hat u_0 \in \check{H}^3$, then the local truncation error $E_n^1$ satisfies the following estimates
\begin{equation}\label{lemma:loc1:e1}
\begin{split}
\tau\sum_{n=1}^{N}\|E_n^1\|\le Q\tau,\qquad
\tau\sum_{n=1}^N\|\delta_{\tau}^2\rho_n\|\le Qh^2.
\end{split}
\end{equation}
\end{lemma}
}
\begin{proof}
We estimate $E_n^1$ in (\ref{Euler:loc1}) as follows:
\begin{equation*}
\begin{split}
\sum_{n=1}^{N}\|E_n^1\|&\le Q\sum_{n=1}^{N}\int_{t_{n-1}}^{t_n}
\|\partial_s^3u(\cdot,s)\|ds\le Q\|u\|_{H^3(L^2)}.
\end{split}
\end{equation*}
For the second statement in (\ref{lemma:loc1:e1}), we have
\begin{equation*}
\begin{split}
\ds\sum_{n=1}^N\|\delta_{\tau}^2\rho_n\|
&\ds=\frac{1}{\tau^2}\sum_{n=1}^N\left\|\int_{t_{n-1}}^{t_n}\partial_{t}\rho dt-\int_{t_{n-2}}^{t_{n-1}}\partial_t\rho dt\right\|\\
&\ds \le \frac{1}{\tau^2}\sum_{n=1}^N\int_{t_{n-1}}^{t_n}\left(\int_{t-\tau}^t\|\partial_y^2 \rho(\cdot, y)\|dy\right) dt\le \frac{2}{\tau}\int_0^T\|\rho_{tt}\|dt.
\end{split}
\end{equation*}
Combining these results with Theorems \ref{thm1:pde} and \ref{thm4} completes the proof.
\end{proof}

 For $t=t_n$,  we approximate $k*\psi$ at
$t=t_n$ by
\begin{equation}\label{LR1}\begin{split}
\ds \int_0^{t_n} k(t_n-s)\psi(\cdot,s)ds = \sum_{k=0}^{n-1}b_{n,k}\psi_k + R_n^1(\psi),
\end{split}
\end{equation}
where
\begin{equation}\label{bnk1}\begin{split}
b_{n,k}:&=\int_{t_k}^{t_{k+1}}k(t_n-s)ds=\frac{(t_n-t_k)^{-\alpha(t_n-t_k)}}{\Gamma(1-\alpha(t_n-t_k))}-\frac{(t_n-t_{k+1})^{-\alpha(t_n-t_{k+1})}}{\Gamma(1-\alpha(t_n-t_{k+1}))}.
\end{split}
\end{equation}

\begin{lemma}\label{L1}
Suppose  $\partial_t\psi\in L^1(L^2)$, then the quadrature error $R_n^1(\psi)$ defined in \eqref{LR1} satisfies the following estimate:
\begin{equation*}
\begin{split}
\tau \sum_{n=1}^{N} \|R_n^1(\psi)\|\le Q \tau \|\psi_t\|_{L^1(L^2)}.
\end{split}
\end{equation*}
\end{lemma}
\begin{proof}
We estimate $R_n^1(\psi)$ as follows:
\begin{equation*}\begin{split}
\sum_{n=1}^N\|R_n^1(\psi)\|&\le \sum_{n=1}^N\left(\sum_{k=1}^n\left\|\int_{t_{k-1}}^{t_k} k(t_n-s)\left(\psi(s)-\psi_{k-1}\right)ds\right\|\right)\\
&\le \sum_{n=1}^N\left(\sum_{k=1}^n\int_{t_{k-1}}^{t_k}|k(t_n-s)|\int_{t_{k-1}}^{t_k}\|\psi_t(t)\|dt ds\right)\\
&\ds=\sum_{k=1}^N\left(\sum_{n=k}^N \int_{t_{k-1}}^{t_k}|k(t_n-s)|ds\right)\int_{t_{k-1}}^{t_k}\|\psi_t(t)\|dt\le Q\int_0^T\|\psi_t\|dt.
\end{split}
\end{equation*}
This completes the proof.
\end{proof}

Based on the estimate (\ref{k1}), we have the following lemma.
\begin{lemma}\label{lem:bnk}
The coefficients $b_{n,k}$ in (\ref{bnk1}) can be bounded by
\begin{equation*}\begin{split}
\sum_{k=0}^{n-1}|b_{n,k}|\le Q_0,\quad \sum_{n=k}^{N}|b_{n,k}|\le Q_0,
\end{split}
\end{equation*}
where $Q_0$ is a constant independent of $N$.
\end{lemma}

The corresponding finite element scheme is to find $U_n\in S_h$ for $1\le n\le N$, such that
\begin{equation}\label{Euler:e3}
\ds (\delta_\tau^2U_n,\chi_h)+a(U_n,\chi_h)=-\sum_{k=0}^{n-1}b_{n,k}a(U_k,\chi_h)+(f_n,\chi_h),\quad \forall \chi_h \in S_h.
\end{equation}
	
\begin{theorem}\label{thm:Euler}
Suppose that Assumption A holds and that $f \in H^2(\check{H}^2) \cap L^2(\check{H}^3)$, $u_0 \in \check{H}^4$, and $\hat u_0 \in \check{H}^3$, then the finite element solution of (\ref{Euler:e3}) satisfies the following error estimate
\begin{equation*}
\max_{1\le n\le N}\|U_n-u_n\|\le Q \hat M(\tau+h^2)
\end{equation*}
for $\hat M:=\|u_{0}\|_{\check H^4}+ \|\hat u_{0}\|_{\check H^3}+ \|f\|_{H^2(\check H^2)} + \|f\|_{L^2(\check H ^3)}$.
\end{theorem}
\begin{proof}
Denote the total error by
$$\ds U_n-u_n=U_n-V_hu_n+V_hu_n-u_n=\theta_n+\rho_n.$$
Substracting (\ref{Euler:e2}) from (\ref{Euler:e3}) and considering the definition of $V_h$ in (\ref{Def:RV}), we get the following error equation
\begin{equation*}
\begin{split}
\ds (\delta_\tau\theta_n,\chi_h)+\tau a(\theta_n,\chi_h)
&=(\delta_\tau\theta_{n-1},\chi_h)-\tau\sum_{k=0}^{n-1}b_{n,k}a(\theta_k,\chi_h)\\
&\ds\quad -\tau G_n^1(\chi_h)-\tau(\delta_\tau^2\rho_n,\chi_h)+\tau(E_n^1,\chi_h),
\end{split}
\end{equation*}
where
\begin{equation*}
\ds G_n^1(\chi_h)= \sum_{k=0}^{n-1} b_{n,k}a(V_hu_k,\chi_h)-\int_0^{t_n}k(t_n-s)a(V_hu(s),\chi_h)ds.
\end{equation*}
We choose $\chi_h=\delta_\tau\theta_n$ to yield
\begin{equation*}
\begin{split}
\|\delta_\tau\theta_n\|^2+ a(\theta_n,\theta_n)=&(\delta_\tau\theta_{n-1},\delta_\tau\theta_n)+a(\theta_n,\theta_{n-1})-\sum_{k=0}^{n-1}b_{n,k}a(\theta_k,\theta_n-\theta_{n-1})\\
&+\tau(E_n^1,\delta_\tau\theta_n)-\tau G_n^1(\delta_{\tau}\theta_n)-\tau(\delta_\tau^2\rho_n,\delta_\tau\theta_n).
\end{split}
\end{equation*}
By applying the Cauchy inequality, we obtain
\begin{equation}\label{thm:Euler:e3}
\begin{split}
\|\delta_\tau\theta_n\|^2+ a(\theta_n,\theta_n)&
\le \|\delta_\tau\theta_{n-1}\|^2+a(\theta_{n-1},\theta_{n-1})
-2\sum_{k=0}^{n-1}b_{n,k}a(\theta_k,\theta_n-\theta_{n-1})\\
&\quad+2\tau(E_n^1,\delta_\tau\theta_n)-2\tau G_n^1(\delta_{\tau}\theta_n) -2\tau(\delta_\tau^2\rho_n,\delta_\tau\theta_n).
\end{split}
\end{equation}
Let $\|\delta_{\tau}\theta_{n^*}\|=\max_{1\le n\le N}\|\delta_{\tau}\theta_n\|$. Summing (\ref{thm:Euler:e3}) over $n=1,\cdots,n^*$ and again using the Cauchy inequality and Lemma \ref{lem:bnk}, we have
\begin{equation*}
\begin{split}
\|\delta_\tau\theta_{n^*}\|^2+a(\theta_{n^*},\theta_{n^*})
&\le 2\sum_{n=1}^{n^*}\sum_{k=0}^{n-1}|b_{n,k}a(\theta_k,\theta_n-\theta_{n-1})|
+2\tau\sum_{n=1}^{n^*}|G_n^1(\delta_{\tau}\theta_n)|\\
&\quad+2\tau\sum_{n=1}^{n^*}\|E_n^1\|\|\delta_\tau\theta_n\|
+2\tau\sum_{n=1}^{n^*}\|\delta_{\tau}^2\rho_n\| \|\delta_\tau\theta_n\|\\
&\le \sum_{n=1}^{n^*}\sum_{k=0}^{n-1}|b_{n,k}|\left( \frac{K}{4Q_0} \|\theta_n-\theta_{n-1}\|_1^2+\frac{4Q_0}{K}\|\theta_k\|_1^2\right)\\
&\quad\ds +2\tau\sum_{n=1}^{n^*}\left(|G_n^1(\delta_{\tau}\theta_n)|
+(\|E_n^1\|+\|\delta_{\tau}^2\rho_n\|)\|\delta_{\tau}\theta_n\|\right)\\
&\le \frac{K}{2} \|\theta_{n^*}\|_1^2+(\frac{4Q_0^2}{K}+K)\sum_{k=1}^{n^*-1}\|\theta_k\|_1^2\\
&\quad\ds +2\tau\sum_{n=1}^{n^*}\left(|G_n^1(\delta_{\tau}\theta_n)|
+(\|E_n^1\|+\|\delta_{\tau}^2\rho_n\|)\|\delta_{\tau}\theta_n\|\right),
\end{split}
\end{equation*}
where we used $(a-b)^2\leq 2(a^2+b^2)$ and
\begin{align*}
   \sum_{n=1}^{n^*}\sum_{k=0}^{n-1}|b_{n,k}| \|\theta_k\|_1^2 & = \sum_{k=0}^{n^*-1} \|\theta_k\|_1^2 \sum_{n=k+1}^{n^*}|b_{n,k}| \leq Q_0\sum_{k=0}^{n^*-1} \|\theta_k\|_1^2 = Q_0\sum_{k=1}^{n^*-1} \|\theta_k\|_1^2.
\end{align*}
By cancelling the term $a(\theta_{n^*},\theta_{n^*})$, and applying the discrete Gronwall inequlity, we arrive at
\begin{equation}\label{thm:Euler:e5}
\begin{split}
\|\delta_{\tau}\theta_{n^*}\|^2
&\le 2\tau\sum_{n=1}^{n^*}|G_n^1(\delta_{\tau}\theta_n)|+2\tau \sum_{n=1}^{n^*} \left(\|E_n^1\|+\|\delta_{\tau}^2 \rho_n\|\right)\|\delta_\tau \theta_n\|.
\end{split}
\end{equation}
Defining the discrete elliptic operator projection $\mathcal L_h$ and the  $L^2-$ orthogonal projection $\mathcal P_h$ by
\begin{equation*}\begin{split}
(\mathcal L_h\phi,\chi_h)=a(\phi,\chi_h),\quad (\mathcal P_h\phi,\chi_h)=(\phi,\chi_h).
\end{split}
\end{equation*}
Then $\mathcal L_h=\mathcal P_h\mathcal L$. Using Lemma \ref{L1}, $G_n^1(\delta_{\tau}\theta_n)$ can be bounded by
\begin{equation}\label{est:Gn1}\begin{split}
\sum_{n=1}^{n^*}|G_n^1(\delta_{\tau}\theta_n)|&\ds=\sum_{n=1}^{n^*}\left|\sum_{k=0}^{n-1}b_{n,k}(\mathcal L_hV_hu_k,\delta_{\tau}\theta_n)
-\int_{0}^{t_n}k(t_n-s)(\mathcal L_hV_hu(s),\delta_{\tau}\theta_n)ds\right|\\
&=\sum_{n=1}^{n^*}\left|(R_n^1(\mathcal L_hV_hu),\delta_{\tau}\theta_n)\right|
\le \sum_{n=1}^{n^*}\|R_n^1(\mathcal L_hV_hu)\| \|\delta_{\tau}\theta_n\|.
\end{split}
\end{equation}
Inserting (\ref{est:Gn1}) into (\ref{thm:Euler:e5}),
 and using Lemma \ref{lemma:loc1}, we get
\begin{equation*}
\begin{split}
\|\delta_{\tau}\theta_{n^*}\|\le Q\tau\sum_{n=1}^{n^*}\left(\|R_n^1(\mathcal L_hV_hu)\|+\|E_n^1\|
+\|\delta_{\tau}^2\rho_n\|\right)\le Q(\tau+h^2),
\end{split}
\end{equation*}
where  Lemma \ref{thm:Vh} is employed to estimate the first term on the right-hand side
\begin{align*}
\sum_{n=1}^{n^*}\|R_n^1(\mathcal L_hV_hu)\| &\leq Q\|(\mathcal L_hV_hu)_t\|_{L^1(L^2)}
 = Q\|\mathcal P_h\mathcal L(V_hu)_t\|_{L^1(L^2)} \\
&\leq Q\|\mathcal L(V_hu_t)\|_{L^1(L^2)}
\leq Q\|u_t\|_{L^1(\check H^2)}.
\end{align*}
Finally, noting that $\|\theta_n\|\le \|\theta_{n-1}\|+\tau\|\delta_{\tau}\theta_n\|$, we have
\begin{equation*}
\|\theta_{n^*}\|\le \|\theta_{0}\|+\tau\sum_{k=1}^{n}\|\delta_\tau\theta_{k}\|\le Q(\tau+h^2).
\end{equation*}
Combining with Lemma \ref{thm:RV1}, we complete the proof.
\end{proof}

\subsection{A fast algorithm}
Let $\{\varphi_j(\bm{x})\}_{j=1}^{M}$ be the nodal basis functions of $S_h$, satisfying $\varphi_j(\bm{x}_j)=1$ and $\varphi_j(\bm{x}_i)=0$ for $i\ne j$, where $M$ denotes the number of degrees of freedom of the finite element space. Denote $\mathbf{M},\mathbf{S}\in R^{M\times M}$ be the mass and stiffness matrices, respectively. Let $U_i^n(1\le i \le M,1\le n \le N)$ be the finite element approximation of (\ref{Euler:e3}) at $(\bm x_i,t_n)$. Define $\mathbf U^{n}=[U_1^n,U_2^n,\cdots,U_M^n]^{\top}$, then the matrix formulation of (\ref{Euler:e3}) is given by
\begin{equation}\label{fMat:e1}
(\mathbf{M}+\tau^2\mathbf{S})\mathbf{U}^n=2\mathbf{M}\mathbf{U}^{n-1}-\mathbf{M}\mathbf{U}^{n-2}-\tau^2\sum_{k=0}^{n-1}b_{n,k}\mathbf{S}\mathbf{U}^k+\tau^2\mathbf{F}^n,\quad n\geq 2,
\end{equation}
where $\mathbf{F}^n=[F_1^n,F_2^n,\cdots,F_M^n]$ with $\ds F_i^n=\int_{\Omega}f(\bm{x},t_n)\varphi_i(\bm{x})d\bm{x}$.  {The initial step $\mathbf U^1$ can be computed by the forward Euler discretization, $\mathbf M\mathbf U^1=\mathbf M\mathbf U^0+\tau \hat{\mathbf  U}^0$, where $\ds\hat{\mathbf{U}}^0_i=\int_{\Omega}\hat u_0\phi_i(\bm x)d\bm x$.}
At each time step, evaluating  $\mathbf{S}\mathbf{U}^k$ requires $O(M)$ for $1\le k \le n-1$ operations, thus the cost of forming the right-hand side of (\ref{fMat:e1}) is $O(nM)$.  Summing over all time levels yields
\begin{equation*}
\ds \sum_{n=1}^NO(Mn)=O(MN^2).
\end{equation*}
Therefore, the overall computational complexity of the time-stepping scheme (\ref{fMat:e1}) is $O(MN^2)$.

To develop a fast algorithm, we denote $$
\begin{array}{l}
\mathbf{U}=[\mathbf{U}^{2,\top},\mathbf{U}^{3,\top},\cdots,\mathbf{U}^{N,\top}]^\top,\\[0.1in]
\mathbf{F}=[\mathbf{F}^{2,\top},\mathbf{F}^{3,\top},\cdots,\mathbf{F}^{N,\top}]^\top,
\end{array}
$$
and reformulate the scheme by solving an all-at-once linear system, instead of updating $(\ref{fMat:e1})$ step by step for $2\le n\le N$:
\begin{equation}\label{fMat:e2}
(\mathbf{E}_{N-1}\otimes \mathbf{M}+\tau^2\mathbf{T}_{N-1}\otimes \mathbf{S})\mathbf{U}=\tau^2\mathbf{F},
\end{equation}
where $\mathbf{E}_{N-1}$ is a banded lower-triangular matrix with bandwidth $3$, whose nonzero entries on the main diagonal and the two lower diagonals are $[1,-2,1]$.
 Moreover, $\bm{T}_{N-1}={\text{toeplitz}}(\bm{t}^c,\bm{t}^r)$ is a $(N-1)\times(N-1)$ sub-triangular Toeplitz matrix with its first column $\bm{t}^c$ defined by
\begin{equation*}
\ds \mathbf{t}_k^c=\left\{ \begin{array}{ll}
	   1,&k=1,\\
\ds \frac{(\tau(k-1))^{-\alpha(\tau(k-1))}}{\Gamma(1-\alpha(\tau(k-1)))}-\frac{(\tau(k-2))^{-\alpha(\tau(k-2))}}{\Gamma(1-\alpha(\tau(k-2)))},&2\le k\le N-1.
\end{array}\right.
\end{equation*}
\begin{theorem}
All-at-once linear system (\ref{fMat:e2}) can be solved in $O(MN\ln^2N)$ operations using a fast divide-and-conquer algorithm.
\end{theorem}
\begin{proof}
We partition $\mathbf{E}_{N-1}$ and $\mathbf{T}_{N-1}$ equally into $2\times2$ block matrices:
\begin{equation}\label{fMat:e4}
\mathbf{E}_{N-1}=
\begin{bmatrix}
\mathbf{E}_{(N-1)/{2}} & 0\\
\mathbf{H}_{(N-1)/2} & \mathbf{E}_{(N-1)/2}
\end{bmatrix},\quad
\mathbf{T}_{N-1}=
\begin{bmatrix}
\mathbf{T}_{{(N-1)}/{2}} & 0\\
\mathbf {L}_{{(N-1)}/{2}} & \mathbf{T}_{{(N-1)}/{2}}
\end{bmatrix},
\end{equation}
where the top-right block of  $\mathbf H_{(N-1)/2}$ is $\begin{bmatrix} 1& -2\\
0& 1\end{bmatrix}$, and all the other entries are zero. The matrix $\mathbf L_{(N-1)/2}=\operatorname{toeplitz}(\mathbf l^c,\mathbf l^r)$ is a $(N-1)/2\times (N-1)/2$ Toeplitz matrix with first row and column given by $\mathbf l_c=\mathbf t^c((N-1)/2+1:N-1), \mathbf l_r=\operatorname{flipud}(\mathbf t^c(2:(N-1)/2+1)) $.
Accordingly, we set $\mathbf{U}=[\mathbf{U}_1^{\top},\mathbf{U}_2^{\top}]^\top$ and $\mathbf{F}=[\mathbf{F}_1^{\top},\mathbf{F}_2^{\top}]^\top$. Then (\ref{fMat:e2}) can be equivalently solved by
\begin{equation}\label{fMat:e5}
\left\{
\begin{array}{l}
(\mathbf{E}_{{(N-1)}/{2}}\otimes\mathbf{M}+\tau^2\mathbf{T}_{{(N-1)}/2}\otimes \mathbf{S})\mathbf{U_1}=\tau^2\mathbf{F_1}\\[0.1in]
(\mathbf{E}_{{(N-1)}/2}\otimes\mathbf{M}+\tau^2\mathbf{T}_{{(N-1)}/{2}}\otimes \mathbf{S})\mathbf{U_2}\\[0.1in]
\quad =\tau^2\mathbf{F_2}-(\mathbf{H}_{{(N-1)}/{2}}\otimes\mathbf{M}+\tau^2\mathbf{L}_{{(N-1)}/2}\otimes \mathbf{S})\mathbf{U_1}.
\end{array}\right.
\end{equation}
From the structure of $\mathbf H$ in (\ref{fMat:e4}), the product $(\mathbf H_{(N-1)/2}\otimes \mathbf M)\mathbf U_1$ requires only $O(MN)$ operations.
Let $\mathbf{X}$ be the matrix form of $\mathbf{U}_1$ which is obtained by reshaping $\mathbf U_1$ into a $M\times{(N-1)}/{2}$ matrix, where the $i$-th row corresponds to the unknowns at $\bm x_i$. With $vec(\cdot)$ denoting the inverse reshaping, the third term on the right-hand side of (\ref{fMat:e5}) can be computed by
\begin{equation*}
(\mathbf L_{(N-1)/2}\otimes \mathbf S)\mathbf{U_1}=vec(\mathbf{S}\mathbf{X}\mathbf L_{{(N-1)}/{2}}^\top)=vec(\mathbf{S}(\mathbf {L}_{{(N-1)}/{2}}\mathbf{X}^\top)^\top).
\end{equation*}
By applying the Fast Fourier Transform, the multiplication $\mathbf L_{(N-1)/{2}}\mathbf X^{\top}$ can be carried out in $O(MN\ln N)$ operations. Consequently, $(\mathbf L_{(N-1)/{2}}\otimes \mathbf S)\mathbf{U}_1$ can be computed in $ O(MNlnN)$ operations. Since the sub-matrices $\mathbf{E}_{(N-1)/{2}}$ and $\mathbf{T}_{(N-1)/{2}}$ share the same structure as $\mathbf E_{N-1}$ and $\mathbf {T}_{N-1}$, we can recursively apply the above partitioning, {which leads} to a {divide-and-conquer} algorithm. With $2^J=N-1$, the total computational complexity of the right-hand side terms is
\begin{equation*}
\ds O(MN\ln N)+2O(\frac{MN}{2}\ln \frac{N}{2})+\cdots+2^JO(\frac{MN}{2^J}\ln \frac{N}{2^J})=O(MN\ln^2N).
\end{equation*}
This {completes} the proof.
\end{proof}

\section{Numerical experiments}\label{secnum}
Throughout this section, we set the final time $T = 1$, $\Omega = (0,1)^d$ with $d = 1$ or $2$
and $K = 0.01$.

\textbf{Example 1}
The exact solution is given by $u(x,t) = t^3 \sin(2\pi x)$  with $\alpha(t) = 1 - \cos(t)$.
The source term is
\begin{equation}\label{f}
\ds f(x,t) = 6t \sin(2\pi x) + 4K\pi^2 \sin(2\pi x)\left(t^3 + \int_{0}^{t} k(s)(t-s)^3\,ds\right),
\end{equation}
in which the convolution term is approximated by the composite Simpson's rule.
We evaluate the $L^2$ error of the numerical solution at $t= T$, and
fix $\displaystyle h = 2^{-7}$ and $\tau=2^{-13}$ to examine the temporal and spatial convergence rates, respectively. The numerical results are presented in Tables~\ref{tab:t1}--\ref{tab:t11},
which demonstrate the first-order temporal accuracy and second-order spatial accuracy of scheme~(\ref{Euler:e3}), in agreement with Theorem~\ref{thm:Euler}.

\begin{table}[H]
\centering
\begin{tabular}{|c|c c|c c|}
\hline
$\tau$ & TSS & $\mathrm{Rate}$ & FDAC & $\mathrm{Rate}$ \\ \hline
$2^{-5}$ & 5.9344e-02 &-- & 5.9344e-02 & --\\ \hline
$2^{-6}$ & 3.0721e-02 & 0.97 & 3.0721e-02 & 0.97 \\ \hline
$2^{-7}$ & 1.5577e-02 & 0.99 & 1.5577e-02 & 0.99 \\ \hline
$2^{-8}$ & 7.7875e-03 & 1.01 & 7.7875e-03 & 1.01 \\ \hline
\end{tabular}
\caption{$L^2$ errors and temporal convergence rates of (\ref{fMat:e1}) and (\ref{fMat:e2}) for Example 1.}
\label{tab:t1}
\end{table}

\begin{table}[H]
\centering
\begin{tabular}{|c|c c|c c|}
\hline
$h$ & TSS & $\mathrm{Rate}$ & FDAC & $\mathrm{Rate}$ \\ \hline
$2^{-3}$ & 1.7477e-02 &-- & 1.7477e-02 &-- \\ \hline
$2^{-4}$ & 4.1395e-03 & 2.08 & 4.1395e-03 & 2.08 \\ \hline
$2^{-5}$ & 1.0222e-03 & 2.02 & 1.0222e-03 & 2.02 \\ \hline
$2^{-6}$ & 2.7006e-04 & 1.92 & 2.7006e-04 & 1.92 \\ \hline
\end{tabular}
\caption{$L^2$ errors and spatial convergence rates of (\ref{fMat:e1}) and (\ref{fMat:e2}) for Example 1.}
\label{tab:t11}
\end{table}


In addition, we  compare the CPU time (denoted as $\mathrm{CPU_{FDAC}}$) required to solve (\ref{fMat:e2})
using the FDAC algorithm with that (denoted as $\mathrm{CPU_{TSS}}$) for solving (\ref{fMat:e1})
via the traditional time-stepping scheme (TSS) under $\displaystyle h=2^{-3}$.
The comparison results are reported in Table~\ref{tab:t2} and Figure~\ref{fig:g11} (left), which clearly show that the FDAC algorithm achieves a much higher computational efficiency.

\begin{table}[H]
\centering
\begin{tabular}{|c|ccccccccccc|}
\hline
$N$ & $2^8$ & $2^9$ & $2^{10}$ & $2^{11}$ & $2^{12}$ & $2^{13}$ & $2^{14}$ & $2^{15}$ & $2^{16}$ & $2^{17}$ & $2^{18}$ \\ \hline
$\mathrm{CPU_{TSS}}$   & 0.8 & 3.0 & 12 & 46 & 202 & 787 & 3251 & 14187 & - & - & - \\ \hline
$\mathrm{CPU_{FDAC}}$  & 0.1 & 0.3 & 0.6 & 1.1 & 2.2 & 4.3 & 8.8 & 19 & 61 & 125 & 248 \\ \hline
\end{tabular}
\caption{CPU times (in seconds) for TSS and FDAC in Example 1.}
\label{tab:t2}
\end{table}

\begin{figure}[H]
\centering
\includegraphics[width=6.5cm]{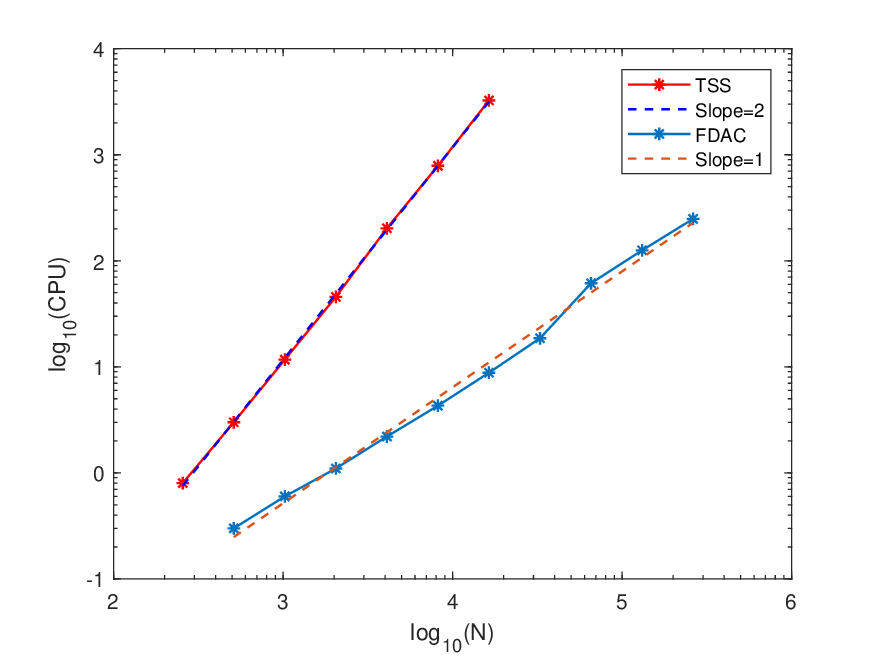}
\includegraphics[width=6.5cm]{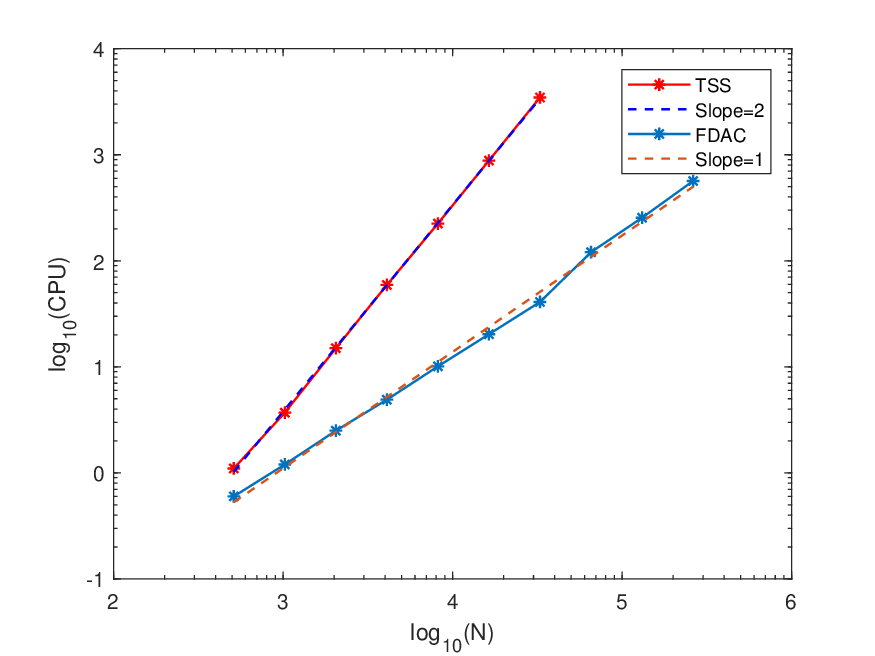}
\caption{CPU times of TSS and FDAC for Example 1 (left) and Example 2 (right) under a log-log scale.}
\label{fig:g11}
\end{figure}

\textbf{Example 2}
Consider the exact solution $u(x,y,t)=t^3 \sin(2\pi x) \sin(2\pi y)$ with $\alpha(t)=t \sin(t)$. We accordingly evaluate the source term as follows
\begin{equation*}
\ds f = 6t \sin(2\pi x)\sin(2\pi y)  + 4K\pi^2 \sin(2\pi x) \sin(2\pi y)\left(t^3 + \int_{0}^{t} k(s)(t-s)^3\,ds\right),
\end{equation*}
and adopt the composite Simpson's rule to approximate the integral term in the above equation.
We fix $\displaystyle h = 2^{-7}$  to examine the temporal convergence rates and
fix $\tau = 2^{-12}$  to test the temporal convergence rates, in which we follow  Example 1 to compute the errors.
The numerical results
presented in Tables ~\ref{tab:t4}--\ref{tab:t41}, which again show the first-order temporal accuracy and second-order spatial accuracy of scheme~(\ref{Euler:e3}). These results substantiate the theoretical findings in Theorem \ref{thm:Euler}.

\begin{table}[]
		\centering
		\begin{tabular}{|c|cc|cc|}
			\hline
			$\tau$ & TSS & $\mathrm{Rate}$ & FDAC & $\mathrm{Rate}$\\ \hline
			$2^{-5}$ & 4.0186e-02 & --  & 4.0186e-02 & -- \\\hline
			$2^{-6}$ & 2.0704e-02 & 0.98  & 2.0704e-02 & 0.98 \\\hline
			$2^{-7}$ & 1.0376e-02 & 1.01  & 1.0376e-02 & 1.01 \\\hline
			$2^{-8}$ & 5.0587e-03 & 1.04  & 5.0587e-03 & 1.04 \\\hline
		\end{tabular}
		\caption{$L^2$ errors and temporal convergence rates of TSS and FDAC for Example 2.}
		\label{tab:t4}
	\end{table}

	\begin{table}[]
		\centering
		\begin{tabular}{|c|cc|cc|}
			\hline
			$ h $ & TSS & $\mathrm{Rate}$ & FDAC & $\mathrm{Rate}$\\ \hline
			$2^{-3}$ & 3.3078e-02 & --  & 3.3078e-02 & -- \\\hline
			$2^{-4}$ & 7.2671e-03 & 2.19  & 7.2671e-03 & 2.19 \\\hline
			$2^{-5}$ & 1.7473e-03 & 2.06  & 1.7473e-03 & 2.06 \\\hline
			$2^{-6}$ & 4.3325e-04 & 2.01  & 4.3325e-04 & 2.01 \\\hline
		\end{tabular}
		\caption{$L^2$ errors and spatial convergence rates of TSS and FDAC for Example 2.}
		\label{tab:t41}
	\end{table}


Furthermore, we compare the CPU time $\mathrm{CPU_{FDAC}}$ for solving (\ref{fMat:e2})
using the FDAC method with $\mathrm{CPU_{TSS}}$ for solving (\ref{fMat:e1}) via the traditional TSS
 under $\displaystyle h=2^{-3}$.
The results in Table~\ref{tab:t5} and Figure~\ref{fig:g11} (right) clearly demonstrate that the FDAC algorithm is significantly more efficient than TSS.

	\begin{table}[H]
		\centering
		\begin{tabular}{|c|c c c c c c c c c c c|}\hline
			$N$ & $2^8$ & $2^9$ & $2^{10}$ & $2^{11}$ & $2^{12}$ & $2^{13}$ & $2^{14}$ & $2^{15}$ & $2^{16}$ & $2^{17}$ & $2^{18}$\\ \hline
			$\mathrm{CPU_{TSS}}$ & 0.3 & 1.1 & 3.7 & 15 & 59  & 224   & 882 & 3467 & 14928 & - & -\\\hline
			$\mathrm{CPU_{FDAC}}$ & 0.3 & 0.6 & 1.2 & 2.5 & 4.9 & 10 & 20 & 41 & 121 & 255 & 566\\\hline
		\end{tabular}
		\caption{CPU times (in seconds) of TSS and FDAC for Example 2.}
		\label{tab:t5}
	\end{table}

\textbf{Example 3}  We set $\alpha(t) = 1 - \cos(t)$, $f\equiv 1$ and $u_0=\hat{u}_0=\sin(\pi x)\sin(\pi y)$.
We fix $h=2^{-5}$ to test the temporal convergence rate while fix $\tau=2^{-5}$ to test its spatial convergence rate.
The temporal errors and its convergence rate are defined as follows
\begin{equation*}\label{temp}\begin{array}{l}
\ds E(\tau,h)=\sqrt{h\sum_{j=1}^{M-1}\bigg|U^j_N(\tau,h)-U^j_{2N}({\tau}/{2},h)\bigg|^2},
 \quad rate^t=\log_2\bigg(\frac{E(2\tau,h)}{E(\tau,h)}\bigg).
\end{array}
\end{equation*}
Similarly, we denote spatial errors and the corresponding convergence rate as follows
\begin{equation*}\label{spat}
G(\tau,h)=\sqrt{h\sum_{j=1}^{M-1}\bigg|U^j_N(\tau,h)-U^{2j}_{N}(\tau,{h}/{2})\bigg|^2},
\quad rate^x=\log_2\bigg(\frac{G(\tau,2h)}{G(\tau,h)}\bigg).
\end{equation*}
	\begin{figure}
			\centering
	\includegraphics[width=5cm]{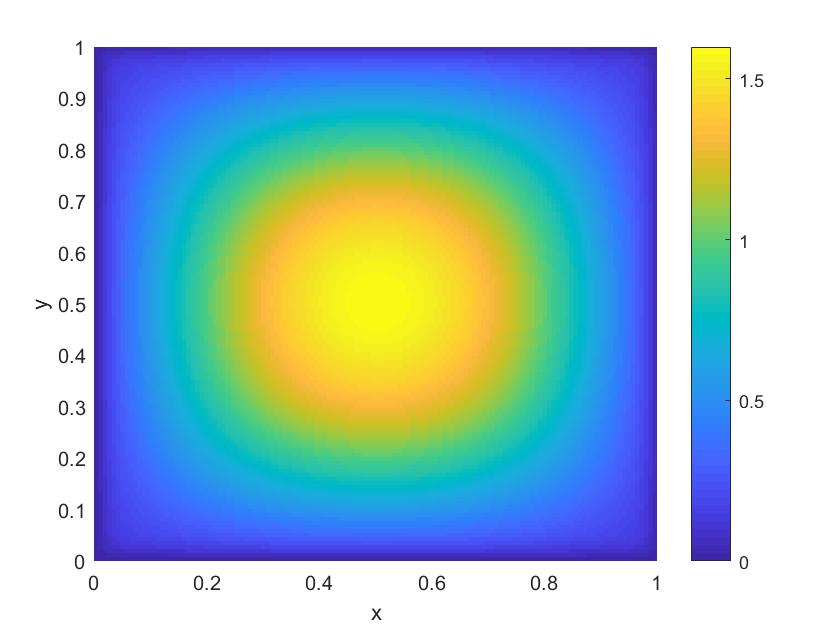}
\includegraphics[width=5cm]{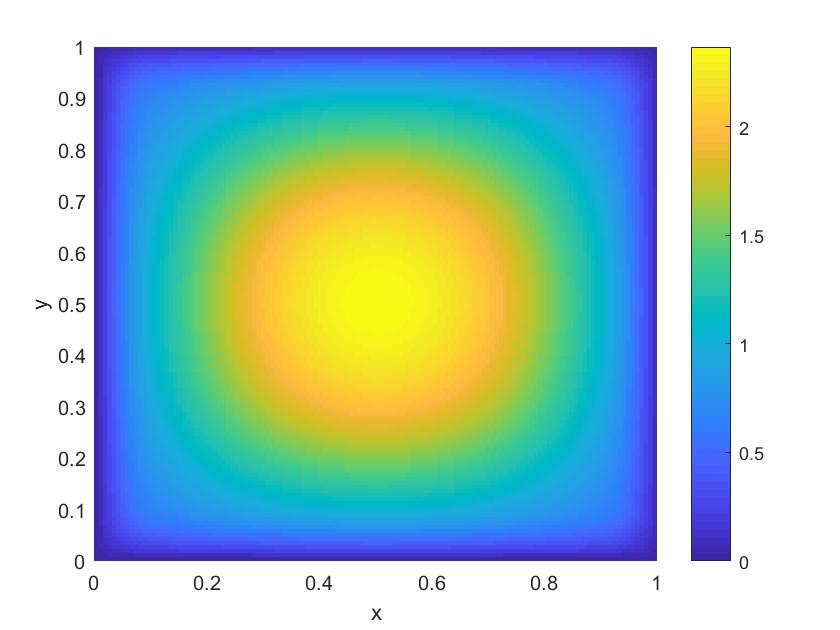}
			\caption{FDAC solutions for Example 3 at $\ds t={T}/2$ (left) and $t=T$ (right).}
			\label{fig:g}
		\end{figure}

We plot the solutions of the FDAC algorithm with  $\tau=2^{-12}$ and $h=2^{-8}$ in Figure \ref{fig:g}, and present the
  numerical results
in Tables \ref{tab:t}--\ref{tab:t1}, which clearly demonstrate the first-order temporal accuracy and and second-order spatial accuracy of scheme~(\ref{Euler:e3}). These results are  consistent with theoretical findings in Theorem~\ref{thm:Euler}.
		\begin{table}[H]
			\centering
			\begin{tabular}{|c|c c|c c|}
				\hline
				$\tau$ & TSS & $rate^{t}$ & FDAC & $rate^{t}$ \\ \hline
				$2^{-5}$ & 1.4291e-02 & -    & 1.4291e-02 & - \\ \hline
				$2^{-6}$ & 7.6061e-03 & 0.91 & 7.6061e-03 & 0.91 \\ \hline
				$2^{-7}$ & 3.9321e-03 & 0.95 & 3.9321e-03 & 0.95 \\ \hline
				$2^{-8}$ & 2.0010e-03 & 0.97 & 2.0010e-03 & 0.97 \\ \hline
			\end{tabular}
			\caption{Errors and temporal convergence rates for Example 3.}
			\label{tab:t}
		\end{table}
		\begin{table}[H]
			\centering
			\begin{tabular}{|c|c c|c c|}
				\hline
				$h$ & TSS & $rate^x$ & FDAC & $rate^x$ \\ \hline
				$2^{-5}$ & 4.0253e-03 & -    & 4.0253e-03 & -    \\ \hline
				$2^{-6}$ & 1.1337e-03 & 1.83 & 1.1337e-03 & 1.83 \\ \hline
				$2^{-7}$ & 3.0186e-04 & 1.91 & 3.0186e-04 & 1.91 \\ \hline
				$2^{-8}$ & 7.6436e-05 & 1.98 & 7.6436e-05 & 1.98 \\ \hline
			\end{tabular}
			\caption{Errors and spatial  convergence rates for Example 3.}
			\label{tab:t1}
		\end{table}

\section*{Acknowledgments}
This work was partially supported by the Postdoctoral Fellowship Program of CPSF (No. GZC20240938), the China Postdoctoral Science Foundation (No. 2024M762459), and the  Natural Science Foundation of Hubei Province (No. 2025AFB109).

\end{document}